\documentclass[11pt,twoside]{amsart}
\usepackage{amsmath}
\usepackage{tikz-cd}
\usepackage{amssymb}
\usepackage{amsthm}
\usepackage{latexsym}
\usepackage{amscd}
\usepackage{xypic}
\xyoption{curve}
\usepackage{ifthen}
\usepackage{hyperref}
\usepackage{graphicx}
\usepackage{enumerate}

 \def\cocoa{{\hbox{\rm C\kern-.13em o\kern-.07em C\kern-.13em o\kern-.15em A}}}

\usepackage{xcolor}

\newtheorem{theorem}{Theorem}[section]
\newtheorem{lemma}[theorem]{Lemma}
\newtheorem{proposition}[theorem]{Proposition}

\theoremstyle{definition}
\newtheorem{remark}[theorem]{Remark}
\newtheorem{definition}[theorem]{Definition}

\newtheorem{notation}[theorem]{Notation}

\newcommand {\Hom}{\mathrm{Hom}}
\newcommand {\Ext}{\mathrm{Ext}}

\newcommand {\ext}{\mathrm{ext}}

\newcommand {\Hilb}{\mathcal{H}\kern -0.25ex{\mathit ilb\/}}

\usepackage[none]{hyphenat}

\newcommand {\fm}{\mathfrak{m}}
\newcommand {\fe}{\mathfrak{e}}
\newcommand {\fb}{\mathfrak{b}}
\newcommand {\fd}{\mathfrak{d}}

\newcommand{\Pic}{\operatorname{Pic}}
\newcommand{\Prior}{\operatorname{Prior}}
\newcommand{\Num}{\operatorname{Num}}

\newcommand{\h}{\operatorname{h}}
\newcommand{\Ho}{\operatorname{H}}

\begin{document}

\title[Higher rank prioritary bundles]
{Higher rank prioritary bundles on ruled surfaces and their global sections}

\author[L.\ Costa, I. Macías Tarrío]{L.\ Costa$^*$, I. Macías Tarrío$^{**}$}

\address{Facultat de Matem\`atiques i Inform\`atica,
Departament de Matem\`atiques i Inform\`atica, Gran Via de les Corts Catalanes
585, 08007 Barcelona, SPAIN } \email{costa@ub.edu}

\address{Facultat de Matem\`atiques i Inform\`atica,
Departament de Matem\`atiques i Inform\`atica, Gran Via de les Corts Catalanes
585, 08007 Barcelona, SPAIN} \email{irene.macias@ub.edu}

\date{\today}
\thanks{$^*$ Partially supported by PID2020-113674GB-I00.}
\thanks{$^{**}$ Partially supported by PID2020-113674GB-I00.}

\subjclass{14J60, 14J26}

\begin{abstract}
Let $X$ be a ruled surface over a nonsingular curve $C$ of genus $g\geq0$. The main goal of this paper is to construct simple prioritary vector bundles of any rank $r$ on $X$ and to give effective bounds for the dimension of their module of global sections. 
\end{abstract}


\maketitle

\tableofcontents
\section{Introduction}
Since the early seventies, the theory of vector bundles on projective varieties has become one of the mainstreams in algebraic geometry. For instance, they offer a valuable point of view in order to better understand the geometry of algebraic varieties. Despite huge achievements in vector bundles theory, many naif questions still remain open. Since the whole category of vector bundles on  a projective variety is usually unwieldy, one has to impose some restrictions on the families of vector bundles under consideration. One of the approaches that has been used recently deals with the cohomological properties of vector bundles. That is, impose some cohomological conditions on the bundles to restrict the class of bundles to study. Nowadays there are many nice papers about vector bundles without intermediate cohomology (ACM bundles) and even a lot of them concern bundles without intermediate cohomology with the maximum number of global sections (Ulrich bundles). The goal of this paper is to study a class of vector bundles, called prioritary bundles, that share a strong cohomological property (see Definition \ref{prioritary_def}). Prioritary sheaves were introduced in the early nineties by Hirschowitz and Laszlo in \cite{Hirschowitz} to study vector bundles on $\mathbb{P}^2$ and, later on, the notion was generalized by Walters in \cite{Walter}.    Apart from the fact that prioritary bundles have a nice cohomological description, they also play an important role in the study of the moduli space of $H$-stable vector bundles.  In fact, under certain conditions on the polarization $H$, $H$-stable vector bundles are prioritary and the moduli space of $H$-stable bundles is an open subset inside the moduli stack of prioritary bundles. Hence, some questions concerning stable bundles can be deduced from the study of the corresponding prioritary bundles. Stability is one of the classical notions used to restrict the class of bundles to consider. For example, prioritary bundles have been used to describe the Picard group of the moduli space of sheaves on quadric surfaces (see \cite{Pedchenko}) and the weak Brill-Noether property (see \cite{coskun}). Hence in some sense prioritary bundles link both approaches.  

In this paper we will pay attention on prioritary bundles on ruled surfaces. It is not difficult to obtain good families of rank two prioritary bundles on algebraic surfaces, but, in general, the difficulties appear when we try to deal with vector bundles of rank $r$ bigger than two, that is when the rank is bigger than the dimension of the base space. On one hand, we will be focused on the  existence of rank $r \geq 3$ prioritary bundles on ruled surfaces  but  we will also be  interested in their module of global sections. In fact, we will be focused on the existence of prioritary bundles with certain number of independent global sections. This study is the first step in the study of subvarieties of Brill-Noether type since effective bounds of the dimension of the module of global sections allows to determine when these subvarieties are non empty (\cite{nuestro}  and \cite{CM2}). Finally let us point out that all the prioritary bundles that we construct are simple. In particular this guarantees that they are indecomposable.

\vspace{3mm}

Now we will outline the structure of the paper.
In Section 2, we will fix some notation and we will recall some basic facts about ruled surfaces and also some properties of prioritary vector bundles on them. We will end the section with a result (Theorem \ref{thprior1}) that will allow us to construct recursively prioritary bundles in subsequent sections. The goal of  Section 3, is to construct rank $3$ simple prioritary vector bundles with many sections. We will use different constructions according to their normalized first Chern class  (see Theorems \ref{prop_ej3}, \ref{th2derk3} and \ref{th3derk3}). 
In Section 4, in Theorems \ref{thrk4''} and \ref{th_rk4} we will construct rank $4$ prioritary vector bundles with many sections, using recursively the results given in the previous section. 
Finally, in Section 5, we will prove (see Theorems \ref{th_rgr1} and \ref{th_rg_r2} ) the existence of simple prioritary  bundles of arbitrarily large rank. To do so, we will construct general families of arbitrary rank prioritary bundles and we will give effective lower bounds for the dimension of their space of global sections. Finally, for the case $r=4$, we will compare the results of this section with the ones of Section 4.

\vspace{3mm}

\noindent {\bf Notation:}
Throughout this paper we will work on an algebraically closed field $K$ of characteristic $0$.

\vspace{0.3cm}
\noindent\textbf{Acknowledgements}. 
The second author expresses gratitude to professor Marian Aprodu for his advice and feedback about this work, and for his hospitality in her stays in Bucharest.

\section{Preliminaries on prioritary sheaves}
The first goal of this section is to fix some notation concerning ruled surfaces and recall basic facts about prioritary bundles on them. We will end the section with a result that allows us to construct simple prioritary bundles recursively.

A ruled surface is a surface $X$, together with a surjective map $\pi: X\rightarrow C$ to a nonsingular curve $C$ of genus $g\geq0$ such that, for every point $y\in C$, the fibre $X_y$ is isomorphic to $\mathbb{P}^1$, and such that $\pi$ admits a section. Recall that a vector bundle $\mathcal{E}$ on $C$ is called normalized if $h^0(C,\mathcal{E})>0$ and  $h^0(C,\mathcal{E}(\mathfrak{d}))=0$ for any divisor $\mathfrak{d}$ on $C$ of negative degree.  A ruled surface is also defined as the projectivization $X=\mathbb{P}(\mathcal{E})$ of a normalized rank $2$ bundle $\mathcal{E}$ on $C$. Let $\mathfrak{e}$ be the divisor on $C$ corresponding to the invertible sheaf $\wedge^2\mathcal{E}$ and let us define $e:=-\deg(\mathfrak{e})$. We will assume $e\geq0$.

Let $C_0\subseteq X$ be a section such that $\pi_*\mathcal{O}_X(C_0) \cong \mathcal{E}$ and let $f$ be a fiber of $\pi$. We have that $\Pic(X)\cong \mathbb{Z}\bigoplus \pi^{*}\Pic(C)$, where $\mathbb{Z}$ is generated by $C_0$. Also $\Num(X)\cong \mathbb{Z}\bigoplus \mathbb{Z}$, generated by $C_0$ and $f$  satisfying $C_0^2=-e$, $C_0\cdot f=1$ and $f^2=0$.
If $\fb$ 
is a divisor on $C$, we will write $\mathfrak{b}f$ instead of $\pi^{*}\mathfrak{b}$. Thus a divisor $D$ on $X$ can be written uniquely as $D= aC_0+\mathfrak{b}f$, being $\fb\in \Pic(C)$, and its element in $\Num(X)$ can be written as $D\equiv aC_0+bf$ with $b=\deg(\mathfrak{b})$.  
The canonical divisor $K_X$ on $X$ is in the class $-2C_0+(\mathfrak{k}+\mathfrak{e})f$, where $\mathfrak{k}$ is the canonical divisor on $C$  of degree $2g-2$.

\begin{notation}

By abuse of notation, we will denote the tensor product $\mathcal{O}_X(aC_0+\fb f)\otimes\mathcal{O}_X(f)$ by $\mathcal{O}_X(aC_0+(\fb+1)f)$ and we will denote the canonical divisor by $K_X=-2C_0-\overline{e}f$ with $\overline{e}:=e-2(g-1)$.
\rm Usually, we will write $\h^0\mathcal{O}_X(D)$ to refer to $\h^0(X,\mathcal{O}_X(D))$.
\end{notation}

For any divisor $D=aC_0+\mathfrak{b} f$ on $X$ with $a\geq0$, it follows from  Lemma V.2.4, Exercises III.8.3 and
III.8.4 of  \cite{hartshorne}, that 
\begin{equation}
\label{igualdadcohom}
    \h^i(X,\mathcal{O}_X(D))=\h^i(C, (S^a\mathcal{E})(\mathfrak{b})),
\end{equation}
 where $S^a\mathcal{E}$ stands for the $a$-th symmetric power of $\mathcal{E}$.

Moreover,
\begin{equation}
    \label{desigcoh}
\h^0(C,\mathcal{O}_C(\mathfrak{b}))\leq \h^0(C, (S^a\mathcal{E})(\mathfrak{b}))\leq \sum_{i=0}^a \h^0(C,\mathcal{O}_C(\mathfrak{b}+i\fe)),  
\end{equation}
for each divisor $\mathfrak{b}$ on $C$ (see for instance \cite[Section 2]{theta}).

\vspace{0.4cm}

For effective divisors on $X$ we have the following 
(see \cite[Lemma 2.2]{nuestro}). 

\begin{lemma}
\label{lema}
Let $X$ be a ruled surface over a  nonsingular curve $C$ of genus $g\geq0$ and $D=aC_0+\mathfrak{b}f$ be a divisor on $X$ with $b:=\deg(\fb)$. If $D$ is effective then $a,b\geq0$.
\end{lemma}

\vspace{0.2cm}
The following is a well-known result.
\begin{proposition}
\label{zgen}
    Let $X$ be a smooth projective surface, $D$ a divisor on $X$ and $Z\subset X$  a generic 0-dimensional subscheme. If $|Z|\geq \h^0\mathcal{O}_X(D)$ then $\h^0I_Z(D)=0$.
\end{proposition}

Given a vector bundle $E$ on $X$ of rank $r$ and fixed Chern Classes $c_i:=c_i(E)\in \Ho^{2i}(X,\mathbb{Z})$, the Riemann-Roch Theorem states that the Euler characteristic of $E$  can be expressed as 
     $$\chi(r;c_1,c_2)=r(1-g)-\frac{c_1K_X}{2}+\frac{c_1^2}{2}-c_2.$$

Prioritary sheaves were introduced on $\mathbb{P}^2$ in the early nineties by Hirschowitz and Laszlo in \cite{Hirschowitz} as a generalization of semistable sheaves. Later on, Walter introduced in \cite{Walter} the notion of prioritary sheaf in a much general context. Let us recall the definition.

\begin{definition} \label{prioritary_def}

If $\pi:X\rightarrow C$ is a  ruled surface and $E$ is a coherent sheaf on $X$, we say that $E$ is prioritary if $\Ext^2(E,E(-f))=0$.
\end{definition}

Since the work of Walter (\cite{Walter}), prioritary sheaves have become a very important class of sheaves, in particular for their relation with the stable ones. In fact, it is well-known that, given $H$ an ample divisor, if $(K_X+f)\cdot H<0$, then $H-$stable sheaves are prioritary. Moreover, there exists an open immersion from the moduli space $M_H(r;c_1,c_2)$ parametrizing rank $r$ $H$-stable vector bundles $E$ on $X$ with $c_1(E)=c_1$ and $c_2(E)=c_2$ to the stack of prioritary sheaves $\Prior_X(r;c_1,c_2)$ and in particular, both spaces have the same expected dimension. 

Ideal sheaves of 0-dimensional schemes are examples of prioritary coherent sheaves. In fact,
\begin{lemma}
\label{lema_IZ}
    Le $X$ be a ruled surface over a non singular curve $C$ of genus $g\geq0$ and let $Z$ be a $0$-dimensional scheme on $X$ of length $|Z|\geq0$ and $D\in\Pic(X)$. Then, $I_Z(D)$ is simple and prioritary. In particular, line bundles are simple and prioritary.
\end{lemma}

\begin{proof}
    Let us consider the exact sequence
    \begin{equation}
    \label{IZ}
        0\rightarrow I_Z(D)\rightarrow \mathcal{O}_X(D)\rightarrow \mathcal{O}_Z(D)\rightarrow0.
    \end{equation}

    First of all we will see that $I_Z(D)$ is simple.
    Applying the functor $\Hom(I_Z(D),-)$ to the exact sequence (\ref{IZ}), we obtain the long exact sequence
    \begin{equation}
    \label{IZ_simple}
0\rightarrow\Hom(I_Z(D),I_Z(D))\rightarrow\Hom(I_Z(D),\mathcal{O}_X(D))\rightarrow\Hom(I_Z(D),\mathcal{O}_Z(D))\rightarrow\cdots
    \end{equation}

    Notice that 
  \[\Hom(I_Z(D),\mathcal{O}_X(D)) \cong  \Ho^2I_Z(K_X) \quad\mbox{and} \quad\Ho^2I_Z(K_X)    \cong K.\]

Hence, by (\ref{IZ_simple}) $\Hom(I_Z(D),I_Z(D))\cong K$ and thus $I_Z(D)$ is simple.

\vspace{0.2cm}
    Now we are going to see that $I_Z(D)$ is prioritary.
    Applying the functor $\Hom(-,I_Z(D-f))$ to (\ref{IZ}), we get the long exact sequence
    \begin{equation}
        \cdots\rightarrow\Ext^2(\mathcal{O}_Z(D),I_Z(D-f))\rightarrow\Ext^2(\mathcal{O}_X(D),I_Z(D-f))\rightarrow\Ext^2(I_Z(D),I_Z(D-f))\rightarrow0
    \end{equation}

    Since $$\Ext^2(\mathcal{O}_X(D),I_Z(D-f))\cong \Ho^2I_Z(-f)\cong \Ho^0\mathcal{O}_X(K_X+f)=0,$$
   we get $\Ext^2(I_Z(D),I_Z(D-f))=0$ and hence $I_Z(D)$ is also prioritary.
    
\end{proof}

\begin{remark}
    Since, for any ample divisor $H$ on $X$ with $(K_X+f)\cdot H<0$, $H$-stable vector bundles are prioritary, we have several examples of rank two simple prioritary vector bundles with sections (see for instance \cite{nuestro}).
\end{remark}

In spite of that, recently, prioritary sheaves have been used to get nice contributions in different problems, there are a lot of questions concerning them still open. In this paper we will focus on their existence and on finding lower bounds for its number of independent sections. To this end, let us finish the section with the following key result that we will apply in subsequent sections.

\begin{theorem}
\label{thprior1}
    Let $X$ be a ruled surface over a nonsingular curve $C$ of genus $g\geq0$. Let  $\mathcal{F}_r$ be the family of rank r vector bundles $E_r$ on $X$ given by a nontrivial extension of type
    \begin{equation}
    \label{eqE3}
 0\rightarrow E_{r-1}\rightarrow E_r\rightarrow \mathcal{O}_X(L)\rightarrow 0    
\end{equation}
  where $L\in \Pic(X)$ and $E_{r-1}$ is a rank $r-1$ vector bundle on $X$ such that $\h^1E_{r-1}(-L)\neq0$.
  \begin{itemize}
  \item [(a)]  If $E_{r-1}$ is a simple vector bundle with $\h^0E_{r-1}(-L)=0$ and $\h^2E_{r-1}(K_X-L)=0$, then $E_{r}$ is simple. 
      \item[(b)]  If $E_{r-1}$ is a prioritary vector bundle with $\h^0E_{r-1}(K_X-L+f)=0$ and $\h^2E_{r-1}(-f-L)=0$, then $E_r$ is prioritary. 
      
  \end{itemize}
  
\end{theorem}

    \begin{proof}

(a) Assume that $E_{r-1}$ is a simple vector bundle with $\h^0E_{r-1}(-L)=0$ and $\h^2E_{r-1}(K_X-L)=0$.
Applying the functor $\Hom(-,E_r)$ to the exact sequence
(\ref{eqE3}),
we obtain
\begin{equation}
    \label{eqsimple}
0\rightarrow \Hom(\mathcal{O}_X(L),E_r)\rightarrow \Hom(E_r,E_r)\rightarrow \Hom(E_{r-1},E_r)\rightarrow\Ext^1(\mathcal{O}_X(L),E_r)\rightarrow\cdots
\end{equation}

Let us first compute $\Hom(\mathcal{O}_X(L),E_r)\cong \Ho^0E_r(-L)$.
If we twist the exact sequence (\ref{eqE3}) by $\mathcal{O}_X(-L)$ and we take cohomology, we obtain $$0\rightarrow \Ho^0E_{r-1}(-L)\rightarrow \Ho^0E_r(-L)\rightarrow \Ho^0\mathcal{O}_X\overset{\alpha}{\rightarrow} \Ho^1E_{r-1}(-L)\rightarrow\cdots$$

Since $\Ho^0E_{r-1}(-L)=0$,  we have   $\Ho^0E_r(-L)\cong \ker(\alpha)$. Notice that the map $\alpha$ can be indentified with the map $$K\rightarrow\Ext^1(\mathcal{O}_X(L),E_{r-1})$$

\noindent which sends $1\in K$ to the nontrivial extension (\ref{eqE3}). Thus, $\alpha$  is an injection and $\Ho^0E_{r}(-L)\cong\ker(\alpha)=0$.

Let us now compute $\Hom(E_{r-1},E_r)$. Applying the functor $\Hom(E_{r-1},-)$ to the exact sequence (\ref{eqE3}), we obtain
$$0\rightarrow \Hom(E_{r-1},E_{r-1})\rightarrow \Hom(E_{r-1}, E_r)\rightarrow \Hom(E_{r-1}, \mathcal{O}_X(L) ) \rightarrow \cdots$$
Since $ \Hom(E_{r-1}, \mathcal{O}_X(L) )\cong \Ext^2(\mathcal{O}_X(L), E_{r-1}(K_X))\cong \Ho^2E_{r-1}(K_X-L)=0$ and $E_{r-1}$ is simple, we obtain $\Hom(E_{r-1}, E_r)\cong \Hom(E_{r-1},E_{r-1}) \cong K$.

Putting altogether, by (\ref{eqsimple}) we get that $1\leq \dim\Hom(E_r, E_r)\leq1$ and hence $E_r$ is simple.

(b) Assume that $E_{r-1}$ is a prioritary vector bundle with $\h^0E_{r-1}(K_X-L+f)=0$ and $\h^2E_{r-1}(-f-L)=0$.

Applying the functor $\Hom(-,E_r(-f))$ to (\ref{eqE3}), we get
\begin{equation}
\label{eqprior}
    \cdots\rightarrow \Ext^2(\mathcal{O}_X(L),E_r(-f))\rightarrow \Ext^2(E_r,E_r(-f))\rightarrow \Ext^2(E_{r-1},E_r(-f))\rightarrow0
\end{equation}

Let us first compute $\Ext^2(\mathcal{O}_X(L),E_r(-f))\cong \Ho^2E_r(-f-L)$.
If we consider the extension (\ref{eqE3}) and we take cohomology, we get  
$$\cdots \rightarrow 
\Ho^2E_{r-1}(-f-L)\rightarrow \Ho^2E_r(-f-L)\rightarrow \Ho^2\mathcal{O}_X(-f)\rightarrow0.$$

Since $ \Ho^2E_{r-1}(-f-L)=0$ and $\Ho^2\mathcal{O}_X(-f)\cong \Ho^0\mathcal{O}_X(K_X+f)=0$, we conclude that 
$\Ho^2E_r(-f-L)=0$.

Now we are going to compute $\Ext^2(E_{r-1}, E_{r}(-f))$. Applying the functor $\Hom(E_{r-1},-)$ to the exact sequence (\ref{eqE3}) twisted by $\mathcal{O}_X(-f)$, we get
$$\cdots \Ext^2(E_{r-1},E_{r-1}(-f))\rightarrow \Ext^2(E_{r-1}, E_r(-f))\rightarrow \Ext^2(E_{r-1},\mathcal{O}_X(L-f))\rightarrow0.$$

Since $E_{r-1}$ is prioritary and $\Ext^2(E_{r-1},\mathcal{O}_X(L-f))\cong \Ho^0E_{r-1}(K_X-L+f)=0$, we get  $ \Ext^2(E_{r-1}, E_r(-f))=0$.

Putting altogether, by (\ref{eqprior}) we get that $E_r$ is prioritary.
    \end{proof}

\section{Rank 3 prioritary bundles}

In the above section we have seen examples of rank 1 and 2 simple prioritary vector bundles with sections.
The main goal of this section is to prove the existence of simple prioritary rank 3 vector bundles on $X$ and to find lower bounds for the dimension of their space of sections.

\vspace{0.2cm}
Notice that $E$ is a simple prioritary vector bundle if and only if, for any divisor $D$ on $X$, $E(D)$ is simple and prioritary. Hence in this section we normalize the first Chern class of our rank $3$ bundle $E$ so that $c_1(E)=s C_0+\mathfrak{t}f$ with $s=0,1,2$.
We will start with the case $c_1=2C_0+\mathfrak{t}f$ and to this end, among others, we will apply  Theorem \ref{thprior1} to construct rank-3 simple prioritary vector bundles with this first Chern class.

\vspace{0.2cm}
First of all let us prove the following technical Lemma.

\begin{lemma}
\label{lema_E2'}
    Let $X$ be a ruled surface over a nonsingular curve $C$ of genus $g\geq0$.  Let us consider $\mathcal{S}_2$ the family of rank-2 vector bundles given by a non trivial extension of type 
    \begin{equation}
        \label{Frecursiva3}
        0\rightarrow \mathcal{O}_X(\fb f)\rightarrow E_2\rightarrow I_Z(C_0+(\fm-\fb) f)\rightarrow0,
    \end{equation}
    where $\fm$ and $\fb$ are divisors on $C$ of degree $m=\deg(\fm )\in\{0,1\}$ and $b=\deg(\fb)\geq1$ and 
     $Z$ is a generic 0-dimensional subscheme of length $|Z|\geq1$. Let $L$ be a divisor on $X$.
    The following holds: \begin{itemize}
        \item[i)] If $-L+\fb f$ is a non effective divisor and $|Z|>\h^0\mathcal{O}_X(-L+C_0+(\fm-\fb)f)$, then $\h^0E_2(-L)=0$.
           \item[ii)] If $L-\fb f$ and $L-C_0-(\fm-\fb)f$ are non effective divisors, then $\h^2E_2(K_X-L)=0$.
          \item[iii)]  $\ext^1(I_Z(C_0+(\fm-\fb)f),\mathcal{O}_X(\fb f))>0$.
          \item[iv)] $E_2$ is a simple prioritary vector bundle.
         \end{itemize}
\end{lemma}

\begin{proof}     Using the exact sequence (\ref{Frecursiva3}) and Proposition \ref{zgen} we directly prove i). The proof of ii) follows using the exact sequence (\ref{Frecursiva3}) and Serre duality.

\noindent iii) Notice that 

$\begin{array}{ll}
\ext^1(I_Z(C_0+(\fm-\fb)f),\mathcal{O}_X(\fb f)) &  \geq -\chi I_Z(C_0+(\fm-2\fb)f+K_X)\\
     & =|Z|-\chi\mathcal{O}_X(-C_0-(\fm-2\fb)f).  
\end{array}$

Since by the Riemann-Roch theorem $\chi\mathcal{O}_X(-C_0-(\fm-2\fb)f)=0$, we get $$\ext^1(I_Z(C_0+(\fm-\fb)f),\mathcal{O}_X(\fb f))\geq |Z|>0.$$

\noindent iv) Since the couple $(\mathcal{O}_X(C_0+(\fm-2\fb)f+K_X), Z)$ satisfies the Cayley-Bacharach property,  $E_2$ is a rank $2$  vector bundle.
Let us now check that $E_2$ is simple. Applying the functor $\Hom(-,E_2)$ to the exact sequence (\ref{Frecursiva3}), we get $$0\rightarrow\Hom(I_Z(C_0+(\fm-\fb)f),E_2)\rightarrow \Hom(E_2,E_2)\rightarrow\Hom(\mathcal{O}_X(\fb f),E_2)\rightarrow\cdots$$

\noindent\textbf{Claim 1: $\Hom(\mathcal{O}_X(\fb f), E_2)\cong\Ho^0 E_2(-\fb f)\cong K$.} 

\noindent\textbf{Proof of Claim 1:}
Twisting the exact sequence (\ref{Frecursiva3}) by $\mathcal{O}_X(-\fb f)$ and taking cohomology, we get $$0\rightarrow \Ho^0\mathcal{O}_X\rightarrow \Ho^0 E_2(-\fb f)\rightarrow\Ho^0 I_Z(C_0+(\fm-2\fb)f)\rightarrow\cdots$$
Since $b\geq1$ and $m\in\{0,1\}$,  $\Ho^0I_Z(C_0+(\fm-2\fb)f)=0$ and hence $\Ho^0 E_2(-\fb f)\cong \Ho^0\mathcal{O}_X\cong K.$

\noindent\textbf{Claim 2:  $\Hom(I_Z(C_0+(\fm-\fb )f), E_2)=0$.}

\noindent\textbf{Proof of Claim 2:}
Applying the functor   $\Hom(I_Z(C_0+(\fm-\fb )f), -)$ to the exact sequence (\ref{Frecursiva3}) we get $$0\rightarrow \Hom(I_Z(C_0+(\fm-\fb )f), \mathcal{O}_X(\fb f))\rightarrow\Hom(I_Z(C_0+(\fm-\fb )f), E_2)\rightarrow$$$$\rightarrow\Hom(I_Z(C_0+(\fm-\fb )f), I_Z(C_0+(\fm-\fb )f)\overset{\phi}{\rightarrow}\Ext^1(I_Z(C_0+(\fm-\fb )f), \mathcal{O}_X(\fb f))\rightarrow\cdots$$

On one hand, 

$\begin{array}{ll}
  \Hom(I_Z(C_0+(\fm-\fb )f), \mathcal{O}_X(\fb f))   
  &\cong\Ext^2(\mathcal{O}_X(\fb f),I_Z(C_0+(\fm-\fb )f+K_X) \\
   & \cong\Ho^2 I_Z(C_0+(\fm-2\fb)f+K_X) \\
   &\cong\Ho^0\mathcal{O}_X(-C_0-(\fm-\fb)f)=0,
\end{array}$

\noindent which implies that $\Hom(I_Z(C_0+(\fm-\fb )f), E_2)\cong \ker(\phi)$. On the other hand, by Lemma \ref{lema_IZ},  $$\Hom(I_Z(C_0+(\fm-\fb )f), I_Z(C_0+(\fm-\fb )f)\cong K,$$ and by v),  $\phi$ is the map that sends $1\in K$ to the non trivial extension $e\in\Ext^1(I_Z(C_0+(\fm-\fb )f), \mathcal{O}_X(\fb f))$ defining $E_2$, which is injective. Hence 
$$\Hom(I_Z(C_0+(\fm-\fb )f), E_2)\cong \ker(\phi)=0.$$

\noindent Putting altogether, $\Hom(E_2,E_2)\hookrightarrow K$ and hence $E_2$ is a simple vector bundle.

Let us now prove that $E_2$ is a prioritary vector bundle.
Applying the functor $\Hom(-,E_2(-f))$ to the exact sequence (\ref{Frecursiva3}), we get $$\cdots\rightarrow \Ext^2(I_Z(C_0+(\fm-\fb)f), E_2(-f))\rightarrow\Ext^2(E_2,E_2(-f))\rightarrow \Ext^2(\mathcal{O}_X(\fb f), E_2(-f))\rightarrow0.$$

\noindent\textbf{Claim 3: $\Ext^2(\mathcal{O}_X(\fb f), E_2(-f))\cong\Ho^2E_2(-(\fb+1)f)=0$.}

\noindent\textbf{Proof of Claim 3:}
Twisting the exact sequence (\ref{Frecursiva3}) by $\mathcal{O}_X(-(\fb+1)f)$ and taking cohomology, we get $$\cdots\rightarrow \Ho^2\mathcal{O}_X(-f)\rightarrow\Ho^2E_2(-(\fb+1)f)\rightarrow\Ho^2I_Z(C_0+(\fm-2\fb-1)f)\rightarrow0,$$

\noindent and, since 

$\begin{array}{ll}
  \Ho^2\mathcal{O}_X(-f)   \cong\Ho^0\mathcal{O}_X(K_X+f)=0 &\\
  \Ho^2I_Z(C_0+(\fm-2\fb-1)f)   \cong \Ho^0\mathcal{O}_X(-C_0-(\fm-2\fb-1)f+K_X)=0,&
\end{array}$

\noindent we get $\Ho^2E_2(-(\fb+1)f)=0$.

\noindent \textbf{Claim 4: $\Ext^2(I_Z(C_0+(\fm-\fb )f),E_2(-f))=0$.}

\noindent \textbf{Proof of Claim 4:}
Applying the functor $\Hom(I_Z(C_0+(\fm-\fb )f,-)$ to the exact sequence (\ref{Frecursiva3}) twisted by $\mathcal{O}_X(-f)$, we get 
 \[\cdots\rightarrow \Ext^2(I_Z(C_0+(\fm-\fb )f),\mathcal{O}_X((\fb-1)f))\rightarrow\Ext^2(I_Z(C_0+(\fm-\fb )f),E_2(-f))\]
 \[\rightarrow\Ext^2(I_Z(C_0+(\fm-\fb )f),I_Z(C_0+(\fm-\fb-1 )f))\rightarrow0\]

Since $\Ext^2(I_Z(C_0+(\fm-\fb )f),\mathcal{O}_X((\fb-1)f))\cong\Ho^0I_Z(C_0+(\fm-2\fb+1)f+K_X)=0$ and, by Lemma \ref{lema_IZ}, $\Ext^2(I_Z(C_0+(\fm-\fb )f),I_Z(C_0+(\fm-\fb-1 )f)=0$, we obtain $\Ext^2(I_Z(C_0+(\fm-\fb )f),E_2(-f))=0$.

Putting altogether, we get $\Ext^2(E_2,E_2(-f))=0$ and hence $E_2$ is a  prioritary vector bundle.

\end{proof}

\begin{theorem}
\label{prop_ej3}
    Let $X$ be a ruled surface over a nonsingular curve $C$ of genus $g\geq0$. Let $c_2\gg0$ be an integer and $c_1=2C_0+\mathfrak{t}f\in \Pic(X)$ with $t=\deg(\mathfrak{t})\leq \min\{-2,-2g+1\}$. Then, there exists a rank 3 simple prioritary vector bundle $E$ on $X$ with $c_1(E)=c_1$ and $c_2(E)=c_2$ such that $\h^0E\geq k:= -t-g$.
\end{theorem}

\begin{proof}
    Let us consider $\mathcal{G}_2$ the family of rank 2 vector bundles given by a non trivial extension of type 
    \begin{equation}
        0\rightarrow \mathcal{O}_X(\fb f)\rightarrow E_2\rightarrow I_Z(C_0-\fb f)\rightarrow0,
    \end{equation}
    where
    $\fb$ is a divisor on $C$ of degree $b=\deg(\fb)=-t-1$ and 
     $Z$ is a generic 0-dimensional subscheme of length $|Z|=c_2+e+1\gg0$.

    Let us now consider the family $\mathcal{G}_3$ of rank-3 vector bundles $E_3$ with $c_1(E_3)=c_1$ and $c_2(E_3)=c_2$ given by a non trivial extension of type 
    \begin{equation}
        \label{Frecursiva3'}
        0\rightarrow E_2\rightarrow E_3\rightarrow\mathcal{O}_X(D)\rightarrow0,
    \end{equation}
    where $D=C_0+\mathfrak{t}f$ and $E_2\in\mathcal{G}_2$.

Notice that by construction, $\h^0E_3\geq\h^0 E_2\geq \h^0\mathcal{O}_X(\fb f)=b+1-g=k=-t-g$.
   
 \noindent\textbf{Claim 1: $\mathcal{G}_3\neq \emptyset$.}

\noindent \textbf{Proof of Claim 1:}
Let us first prove that the dimension of the space of extensions of type (\ref{Frecursiva3'}) is positive.
Since $$\ext^1(\mathcal{O}_X(D), E_2)=\h^1E_2(-D)\geq -\chi E_2(-D),$$
it is enough to prove that $\chi E_2(-D)<0$.

Since
$$c_1(E_2(-D))=c_1(E_2)-2D=-C_0-2\mathfrak{t}f$$ and  $$c_2(E_2(-D))=c_2(E_2)-c_1(E_2)\cdot D+D^2=|Z|+b+t=|Z|-1,$$
we get $$\chi E_2(-D)=2-g-|Z|<0,$$ where the last inequality follows from the fact that $
|Z|\gg0$.
Hence the dimension of the extension of type (\ref{Frecursiva3'}) is positive.
By construction, $E_3$ is a rank-3 vector bundle with $c_1(E_3)=C_0+D=2 C_0+\mathfrak{t}f$ and $c_2(E_3)=c_2(E_2)+D\cdot c_1(E_2)=|Z|-1-e=c_2.$
Hence $\mathcal{G}_3\neq \emptyset$.


\noindent\textbf{Claim 2: $E_3$ is simple and prioritary}

\noindent\textbf{Proof of Claim 2:}
By Lemma \ref{lema_E2'}, $E_2$ is a simple vector bundle and thus, by Theorem \ref{thprior1}, it is enough to see that $\h^0E_2(-D)=0$ and $\h^2E_2(K_X-D)=0$.

Since
$-D+\fb f=-C_0+(\fb -\mathfrak{t})f$ is a non effective divisor and $$\h^0\mathcal{O}_X(-D+C_0-\fb f)=\h^0\mathcal{O}_C(1)<|Z|,$$  it follows from Lemma \ref{lema_E2'} that $\Ho^0E_2(-D)=0$.
Analogously, since $D-\fb f=C_0-(2\mathfrak{t}+1)f$ and $D-C_0+\mathfrak{b}f=-f$ are non effective divisors, by Lemma \ref{lema_E2'} we get $\h^2E_2(K_X-D)=0$. Hence, by Theorem \ref{thprior1}, $E_3$ is simple.

By Lemma \ref{lema_E2'}, $E_2$ is a prioritary vector bundle and thus, by Theorem \ref{thprior1},  it is enough to see that $\h^0E_2(K_X+f-D)=0$ and $\h^2E_2(-f-D)=0$.

Since
 $$K_X+f+\fb f-D=-3C_0+(-\mathfrak{t}+\fb +1-2g+2+e)f$$ and $$K_X+f-D+C_0-\fb f=-2C_0+(-\mathfrak{t}-\fb+2g-1-e)f$$ are non effective divisors, it follows from Lemma \ref{lema_E2'} that $\Ho^0E_2(K_X+f-D)=0$.
 Following the same arguments, we see that $\Ho^2E_2(-f-D)=0$. 
 Hence, by Theorem \ref{thprior1},  $E_3$ is prioritary and, therefore,  $E_3$ is a simple prioritary vector bundle.

\end{proof}

To study the other cases of $c_1$, we will contruct rank-3 simple prioritary vector bundles using a different method.
Let us start proving some preparatory Lemmas.

\begin{lemma}
\label{lema_previo1}
Let $X$ be a ruled surface over a nonsingular curve $C$ of genus $g\geq0$  and $m=0,1$.
    Let $E_2$ be a rank 2 vector bundle on $X$  given by a of nontrivial extension of type
   \begin{equation}
       \label{fam_F1}
 0\rightarrow\mathcal{O}_X(-C_0+(\fb+1)f)\rightarrow E_2\rightarrow I_Z((m+1)C_0+(\mathfrak{d}-2\fb-1)f)\rightarrow0
\end{equation}
where $\mathfrak{d},\fb\in \Pic(C)$ are divisors of degree $d=\deg(\mathfrak{d})\geq0$ and $$b=\deg(\fb )>\max\{\frac{d-1}{3},\frac{d+2(g-1)-e}{3}\}$$  and $Z$ is a  0-dimensional subscheme of length $|Z|\geq0$. 
 Then:
 \begin{itemize}
     \item[(a)] $E_2$ is simple, 
     \item[(b)] $\Ho^2E_2(K_X-\fb f)=0$, 
     \item[(c)]$\Ho^0E_2(K_X-(\fb -1)f)=0$, 
     \item[(d)]$\Ho^0E_2(K_X+C_0-\fb f)=0$, 
    \item[(e)]$\Ho^0E_2(-\fb f)=0$.
 \end{itemize} 
\end{lemma}

\begin{proof}

(a) To see that $E_2$ is simple, we apply the functor $\Hom(-,E_2)$ to the exact sequence (\ref{fam_F1}) and we get
\begin{equation}
    \label{simple2'}
    \begin{split}
        0\rightarrow \Hom(I_Z((m+1)C_0+(\mathfrak{d}-2\fb-1)f),E_2)\rightarrow \Hom(E_2,E_2)\rightarrow\\
        \rightarrow \Hom(\mathcal{O}_X(-C_0+(\fb+1)f, E_2)\rightarrow\cdots
    \end{split}
\end{equation}

\vspace{0.2cm}
Let us first prove that $\Hom(\mathcal{O}_X(-C_0+(\fb+1)f),E_2)\cong K$.
Notice that $$\Hom(\mathcal{O}_X(-C_0+(\fb +1)f,E_2)\cong \Ho^0E_2(C_0-(\fb+1) f).$$
If we twist the exact sequence (\ref{fam_F1}) by $\mathcal{O}_X(C_0-(\fb +1)f)$ and we take cohomology, we get
$$0\rightarrow \Ho^0\mathcal{O}_X\rightarrow \Ho^0E_2(C_0-(\fb+1)f)\rightarrow \Ho^0I_Z((m+2)C_0+(\mathfrak{d}-3\fb-2)f)\rightarrow\cdots$$

 Since $3b>d-1$, the divisor $(\fm+2)C_0+(\mathfrak{d}-3\fb-2)f$ is non effective and thus $\Ho^0I_Z((m+2)C_0+(\mathfrak{d}-3\fb -2)f)=0$, which implies that $\Ho^0E_2(C_0-(\fb+1) f)\cong K$.

\vspace{0.2cm}
Let us now see  that $\Hom(I_Z((m+1)C_0+(\fd-2\fb-1)f),E_2)=0$. Denote by $\overline{D}$ the divisor $(m+1)C_0+(\fd-2\fb-1)f$.
Applying the functor $\Hom(I_Z(\overline{D}),-)$ to the exact sequence (\ref{fam_F1}) we get
\begin{equation}
    \label{simple3'}
    \begin{split}
    0\rightarrow \Hom(I_Z(\overline{D}),\mathcal{O}_X(-C_0+(\fb +1)f))
    \rightarrow \Hom(I_Z(\overline{D}),E_2)\rightarrow\\
    \Hom(I_Z(\overline{D}),I_Z(\overline{D}))
    \overset{\alpha}{\rightarrow} \Ext^1(I_Z(\overline{D}),\mathcal{O}_X(-C_0+(\fb+1)f))\rightarrow\cdots
    \end{split}
\end{equation}

Since $-(m+2)C_0+(3\fb-\fd+2)f$  is non effective, 

$\begin{array}{ll}
    \Hom(I_Z(\overline{D}),\mathcal{O}_X(-C_0+(\fb+1) f)) &  \cong \Ext^2(\mathcal{O}_X(-C_0+(\fb+1) f),I_Z(\overline{D}+K_X)\\
   &\cong \Ho^2I_Z((m+2)C_0-(3\fb-\fd+2)f+K_X) \\
   & \cong \Ho^0\mathcal{O}_X(-(m+2)C_0+(3\fb-\fd+2)f)=0.
    
\end{array}$

On the other hand, by Lemma \ref{lema_IZ}, $\Hom(I_Z(\overline{D}),I_Z(\overline{D})\cong K.$

Hence, by (\ref{simple3'}) we get $$\Hom(I_Z(\overline{D}),E_2)\cong \ker (\alpha)=0,$$
where the last equality follows from the fact that $\alpha$ is the map that sends $1$ to the nontrivial extension $e\in \Ext^1(I_Z(\overline{D}),\mathcal{O}_X(-C_0+(\fb+1)f))$ and thus it is injective.

Finally, by (\ref{simple2'}), $1\leq\dim\Hom(E_2,E_2)\leq1$ and hence $E_2$ is simple.  

\vspace{0.2cm}
The proof of (b)-(e) is straightforward using the exact sequence (\ref{fam_F1}) and the lower bounds of $b$.

\end{proof}

\begin{lemma}
\label{lema_privio2}
Let $X$ be a ruled surface over a nonsingular curve $C$ of genus $g\geq0$.
     Let $E_3$ be a rank 3 vector bundle on $X$ given by a nontrivial extension 
\begin{equation}  \label{fam_E3}
 e: 0\rightarrow \mathcal{O}_X(\fb f)\rightarrow E_3\rightarrow E_2\rightarrow0
\end{equation}
where $E_2$ and $\fb$ are given as in Lemma \ref{lema_previo1}. Then, $E_3$ is simple and prioritary.
\end{lemma}
\begin{proof}

First of all we will see that $E_3$ is simple. To this end,
    applying the functor $\Hom(-,E_3 )$ to the exact sequence (\ref{fam_E3}), we get
    \begin{equation}
       \label{simple5} 
       0\rightarrow \Hom(E_2, E_3)\rightarrow \Hom(E_3, E_3)\rightarrow \Hom(\mathcal{O}_X(\fb f), E_3)\rightarrow\cdots
    \end{equation}

\noindent Twisting the exact sequence (\ref{fam_E3}) by $\mathcal{O}_X(-\fb f)$ and taking cohomology, we get
$$0\rightarrow \Ho^0\mathcal{O}_X\rightarrow \Ho^0E_3(-\fb f)\rightarrow \Ho^0 E_2(-\fb f)\rightarrow\cdots$$

\noindent By Lemma \ref{lema_previo1}, $\Ho^0 E_2(-\fb f)=0$, which implies that $$\Hom(\mathcal{O}_X(\fb f), E_3)\cong\Ho^0E_3(-\fb f)\cong \Ho^0\mathcal{O}_X\cong K.$$

On the other hand, $\Hom(E_2, E_3)=0$. In fact, since by  Lemma \ref{lema_previo1}, $$\Hom(E_2,\mathcal{O}_X(\fb f)\cong \Ho^2E_2(K_X-\fb f)=0,$$
it follows from the fact that by Lemma \ref{lema_previo1}, $E_2$ is simple and we are considering non-trivial extensions. Thus, 
using (\ref{simple5})  we conclude that $E_3$ is simple.

Now we are going to see that $E_3$ is prioritary.
Applying the functor $\Hom(E_3,-)$ to the exact sequence (\ref{fam_E3}) twisted by $\mathcal{O}_X(-f)$, we get
\begin{equation}
    \cdots\rightarrow \Ext^2(E_3, \mathcal{O}_X((\fb-1) f))\rightarrow \Ext^2(E_3, E_3(-f))\rightarrow \Ext^2(E_3,E_2(-f))\rightarrow0.
\end{equation}

First of all, let us see that $\Ext^2(E_3, \mathcal{O}_X((\fb-1) f))=0$.
Notice that $$\Ext^2(E_3, \mathcal{O}_X((\fb-1) f))\cong  \Ho^0E_3(K_X-(\fb-1) f).$$

Since $K_X+f$ is non effective and by Lemma \ref{lema_previo1},  $\Ho^0E_2(K_X-(\fb-1) f)=0$, using the exact sequence (\ref{fam_E3}) we get $\Ho^0E_3(K_X-(\fb-1) f)=0$.

Let us now see that $\Ext^2(E_3,E_2(-f))=0$. By definition, $E_2$ is given by a nontrivial extension as in (\ref{fam_F1}).
Applying the functor $\Hom(E_3,-)$ to it twisted by $\mathcal{O}_X(-f)$, we get
$$\cdots\rightarrow \Ext^2(E_3,\mathcal{O}_X(-C_0+\fb f))\rightarrow \Ext^2(E_3, E_2(-f))\rightarrow \Ext^2(E_3, I_Z((m+1)C_0+(\mathfrak{d}-2\fb-2)f))\rightarrow0.$$

By Lemma \ref{lema_previo1} and the fact that $K_X+C_0$ is non effective, we deduce that
$$\Ext^2(E_3,\mathcal{O}_X(-C_0+\fb f))\cong \Ho^0E_3(K_X+C_0-\fb f)=0.$$

Let us now prove that $$\Ext^2(E_3, I_Z((m+1)C_0+(\mathfrak{d}-2\fb-2)f))=0.$$
If we apply the functor $\Hom(-, I_Z((m+1)C_0+(\mathfrak{d}-2\fb-2)f))$ to the exact sequence (\ref{fam_E3}) we get
$$\cdots\rightarrow \Ext^2(E_2, I_Z((m+1)C_0+(\mathfrak{d}-2\fb-2)f))\rightarrow \Ext^2(E_3, \mathcal{O}_X((m+1)C_0+(\mathfrak{d}-2\fb-2)f))$$
$$\rightarrow \Ext^2(\mathcal{O}_X(\fb f), I_Z((m+1)C_0+(\mathfrak{d}-2\fb-2)f))\rightarrow0.$$

Since $K_X-(\fm+1)C_0-(\mathfrak{d}-3\fb-2)f$ is not effective, by duality, 

$$ \Ext^2(\mathcal{O}_X(\fb f), I_Z((\fm+1)C_0+(\mathfrak{d}-2\fb-2)f))  =0 $$


On the other hand, applying the functor $\Hom(-,I_Z((m+1)C_0+(\fd-2\fb-2)f))$ to the exact sequence (\ref{fam_F1}) and using the fact that $I_Z$ is prioritary we get $$\Ext^2(E_2, I_Z((m+1)C_0+(\fd-2\fb-2)f))=0$$ 
and hence $\Ext^2(E_3,E_2(-f))=0$.
Therefore, $E_3$ is also prioritary.

\end{proof}

Let us first consider the case $c_1=C_0+\mathfrak{d}f$.

\begin{theorem}
\label{th2derk3}
Let $X$ be a ruled surface over a non singular curve $C$ of genus $g\geq0$. Let us consider $c_2>>0$ an integer,  $\mathfrak{d}\in Pic(C)$ with $d=\deg(\mathfrak{d})\geq0$. Let $l$, $0\leq l<5$, be an integer equivalent to $c_2-2e-3+d$ module 5.
Then, there exists a rank-3 simple prioritary vector bundle $E$ with $c_1(E)=C_0+\mathfrak{d}f$ and $c_2(E)=c_2$, such that $$\h^0 E\geq k:=\frac{1}{5}(c_2-2e-3+d-l)+1-g.$$
\end{theorem}

\begin{proof}  
Let us consider $\mathcal{F}_2$ the family of rank two vector bundles $E_2$ on $X$ given by  a nontrivial extension of type
\begin{equation}
\label{F2}
  0\rightarrow\mathcal{O}_X(-C_0+(\fb +1)f)\rightarrow E_2\rightarrow I_Z(2C_0+(\mathfrak{d}-2\fb-1)f)\rightarrow0  
\end{equation}

\noindent where $\fb\in\Pic(C)$ has degree $b=\frac{1}{5}(c_2-2e-3+d-l)$ and $Z$ is a 0-dimensional subscheme of length $|Z|=l$.

\vspace{0.2cm}
Since $c_2\gg0$,  we can assume without loss of generality that $$b>\max\{0,2(g-1),\frac{d-1}{3},\frac{d+2(g-1)-e}{3}\}.$$

\noindent\textbf{Claim 1: $\mathcal{F}_2\neq\emptyset$.}

\noindent\textbf{Proof of Claim 1:} First of all we will see that the space of extensions has positive dimension. Notice that
 $$\ext^1:=\ext^1(I_Z(2C_0+(\mathfrak{d}-2\fb-1)f),\mathcal{O}_X(-C_0+(\fb +1)f))=\h^1I_Z(3C_0+(\mathfrak{d}-3\fb-2)f+K_X)$$
 $$\geq -\chi I_Z(3C_0+(\mathfrak{d}-3\fb-2)f+K_X)=|Z|-\chi\mathcal{O}_X(3C_0+(\mathfrak{d}-3\fb-2)f+K_X).$$

Since  $c_2>>0$, 
$$\chi\mathcal{O}_X(-3C_0-(\mathfrak{d}-3\fb-2)f)=-6+2g-3e+2d-6b<0$$ and thus
 $$\ext^1=l+6-2g+3e-2d+6b>0.$$ Therefore the space of extensions of type (\ref{F2}) is nonempty.

\vspace{0.2cm}
On the other hand, since the divisor $3C_0+(\mathfrak{d}-3\fb-2)f+K_X$ is not effective  by the lower bound $3b>d-1$, $Z$ satisfyies the Cayley-Bacharach property for the linear system $|3C_0+(\mathfrak{d}-3\fb-2)f+K_X|$. Hence,  $E_2$ is a rank 2 vector bundle with $c_1(E_2)=C_0+(\mathfrak{d}-\mathfrak{b})f$ and $c_2(E_2)=|Z|+(2e+4b+3-d)=l+2e+4b+3-d$.

Putting altogether, $\mathcal{F}_2$ is nonempty.
\vspace{0.2cm}


Let us now consider $\mathcal{F}_3$ the family of rank 3 vector bundles $E_3$ on $X$ with $c_1(E_3)=C_0+\fd f$ and $c_2(E_3)=c_2$ given by a non trivial extension of type
\begin{equation}
    \label{F3}
    0\rightarrow \mathcal{O}_X(\fb f)\rightarrow E_3\rightarrow E_2\rightarrow0,
\end{equation}
where $E_2\in \mathcal{F}_2$.

\noindent\textbf{Claim 2: $ \mathcal{F}_3\neq \emptyset$.}

\noindent\textbf{Proof of Claim 2:}
By construction, $E_3$ is a rank 3 vector bundle on $X$ with $c_1(E_3)= C_0+\fd f$ and $c_2(E_3)=c_2$.
Notice that $$s:=\ext^1(E_2,\mathcal{O}_X(\fb f))=\ext^1(\mathcal{O}_X(\fb f), E_2(K_X))=\h^1E_2(K_X-\fb f)\geq-\chi E_2(K_X-\fb f)$$
 and, by Serre duality, $-\chi E_2(K_X-\fb f)=-\chi E^*_2(\fb f)$.

\noindent Since

$\begin{array}{ll}
   c_1(E^*_2(\fb f))= & -c_1(E_2)+2\fb f=-C_0+(3\fb-\mathfrak{d})f  \thinspace \thinspace\text{and} \thinspace  \\
   c_2(E^*_2(\fb f))= & c_2(E_2)+(\fb f)c_1(E^*_2)+(\fb f)^2=2e-d+3b+3+l,
\end{array}$

\noindent by Riemann-Roch Theorem we have
$$\chi E^*_2(\fb f)=2(1-g)-\frac{1}{2}(-C_0+(3\fb-\mathfrak{d})f)\cdot K_X+\frac{1}{2}(-C_0+(3\fb-\mathfrak{d})f)^2-(2e-d+3b+3+l))$$
$$=l+2e-d+3b+3.$$

Hence $s>0$ and $\mathcal{F}_3$ is non empty.

\vspace{0.2cm}

It follows from Lemma \ref{lema_privio2}  that $E_3$ is a simple  prioritary vector bundle. 
Moreover, by (\ref{F3}) and the fact that $b>\max\{0,2(g-1)\}$, we have $$\h^0E_3\geq \h^0\mathcal{O}_X(\fb f)=\h^0\mathcal{O}_C(\fb)=b+1-g=\frac{1}{5}(c_2-2e-3+d)+1-g.$$ 
    
\end{proof}

Finally, we will deal with the case $c_1=\fd f$.

\begin{theorem}
\label{th3derk3}
Let $X$ be a ruled surface over a nonsingular curve $C$ of genus $g\geq0$. Let $c_2>>0$ be an integer and $\mathfrak{d}\in \Pic(C)$ such that $d=\deg(\mathfrak{d})\geq0$. 
Then, there exists a rank-3 simple prioritary vector bundle with $c_1(E)=\mathfrak{d}f$ and $c_2(E)=c_2$ such that $$\h^0E\geq \frac{1}{3}(c_2+d+1-3g-e-l)+1-g,$$ where $l$ is an integer $1\leq l\leq3$ equivalent to $c_2+d+1-3g-e$ modulo $3$.
\end{theorem}

\begin{proof}
The proof follows step by step as  in the proof of Theorem \ref{th2derk3}. In this case we consider the family $\mathcal{G}_2$ of rank two vector bundles $E_2$ on $X$  given by a nontrivial extension of type
\begin{equation}0\rightarrow\mathcal{O}_X(-C_0+(\fb+1)f)\rightarrow E_2\rightarrow I_Z(C_0+(\mathfrak{d}-2\fb-1)f)\rightarrow0
\label{fam_G1}
\end{equation}
where $b=\deg(\fb)=\frac{1}{3}(c_2+d-2-e-l)$ and $Z$ is a 0-dimensional subscheme of length $|Z|=l$ and 
 the family $\mathcal{G}_3$ of rank 3 vector bundles $E_3$ on $X$ given by a non trivial extension of type
     \begin{equation}
        \label{eqG2}0\rightarrow \mathcal{O}_X(\fb f)\rightarrow E_3\rightarrow E_2\rightarrow0, 
     \end{equation}
where $E_2\in\mathcal{G}_2$.
It can be seen that $E_3$ is a simple prioritary vector bundle with $c_1(E)=\mathfrak{d}f$, $c_2(E)=c_2$ and $\h^0E\geq \frac{1}{3}(c_2+d+1-3g-e-l)+1-g.$

\end{proof}

\section{Rank 4 prioritary bundles}

The goal of this section is to prove the existence of simple prioritary rank $4$ vector bundles and bound from below the dimension of their space of sections. To this end,
we can also apply Theorem \ref{thprior1} and the families constructed in Section 3. 
First of all, in order to simplify the proof of some results in this section, we summarize in the following two Lemmas some easy cohomological facts.

\begin{lemma}
\label{lema_recursivo}
    Let $X$ be a ruled surface of genus $g\geq0$, $c_2\gg0$ an integer and $D\in \Pic(X)$ a  divisor such that $(-1)^\alpha D$ and $K_X+f+2(-1)^\alpha D$ are  non effective divisors for $\alpha\in\{0,1\}$. 
    Let us consider $\mathcal{S}_3$ the family of rank-3 vector bundles given by non trivial extensions of type  \begin{equation}
        0\rightarrow E_2\rightarrow E_3\rightarrow \mathcal{O}_X(D)\rightarrow0
        \label{E3}
    \end{equation}
    where $E_2$ is a rank-2 vector bundle on $X$ such that $\Ho^1E_{2}(-D)\neq0$.
    The following holds:
   \begin{itemize}
        \item[i)] If $\h^0E_2(D)=0$ then $\h^0E_3(D)=0$.
         \item[ii)] If $\h^0E_2(K_X+f+D)=0$ then $\h^0E_3(K_X+f+D)=0$.
         
           \item[iii)] If $\h^2E_2(K_X+D)=0$ then $\h^2E_3(K_X+D)=0$.
          \item[iv)] If $\h^2E_2(-f+D)=0$ then $\h^2E_3(-f+D)=0$.
    \end{itemize}
    
\end{lemma}

\begin{proof}

 It is straightforward using (\ref{E3}).

\end{proof}

\begin{lemma}
    \label{lemma_E2}
    Let $X$ be a ruled surface over a non singular curve $C$ of genus $g\geq0$ and let $L\in \Pic(X)$. Let us consider the family $\mathcal{H}_2$ of rank-2 vector bundle on $X$ given by a non-trivial extension of type
    \begin{equation}
        \label{primeraextension_E2}
        \mathcal{H}_2: 0\rightarrow \mathcal{O}_X\rightarrow E_2\rightarrow I_Z(f)\rightarrow0,
    \end{equation}
    where $Z$ is a generic 0-dimensional scheme of length $|Z|\geq1$. 
    The following holds: \begin{itemize}
        \item[i)] If $-L$ is a non effective divisor and $|Z|>\h^0\mathcal{O}_X(-L+f)$,  then $\h^0E_2(-L)=0$.
           \item[ii)] If $L$ is a non effective divisor, then $\h^2E_2(K_X-L)=0$.
         
    \end{itemize}
 \end{lemma}
 \begin{proof}
 It is straightfoward using (\ref{primeraextension_E2}).

 \end{proof}

Arguing as in Lemma \ref{lema_E2'}, we get the following result.
\begin{lemma}
\label{prop_ej1}
    Let $X$ be a ruled surface over a nonsingular curve $C$ of genus $g\geq0$ and $Z$ be a generic 0-dimensional scheme of length $|Z|\geq1$. 
   Let  $D$ be a divisor such that 
\begin{itemize}
    \item[i)]  $-D$ is  non effective and $|Z|>\h^0\mathcal{O}_X(-D+f)$.
           \item[ii)] $D$ is a non effective divisor.
       \end{itemize}

        \noindent  Let  $\mathcal{F}_3$ be the family of rank-3 vector bundles given by non trivial extensions of type 
     $$  \mathcal{F}_3: 0\rightarrow E_2\rightarrow E_3\rightarrow \mathcal{O}_X(D)\rightarrow0$$
    where $E_2\in\mathcal{H}_2$ of Lemma \ref{lemma_E2} and such that $\h^1 E_2(-D)>0$. 
         
    Then, $E_3$ is a simple prioritary vector bundle.
\end{lemma}

First of all, we will focus the attention on rank $4$ bundles $E$ on $X$ with $c_1(E)=f$. In this case we use two different constructions according to the sign of the invariant of $X$, $\overline{e}:=e-2(g-1)$.

   \begin{theorem}
   \label{thrk4''}
    Let $X$ be a ruled surface over a non singular curve $C$ of genus $g\geq0$ such that $\overline{e}:=e-2(g-1)\geq0$ and let $c_2\gg0$ be an integer. Then, there exists a rank-4 simple prioritary vector bundle $E$ with $c_1(E)=f$ and $c_2(E)=c_2$ such that $\h^0E\geq1$.
\end{theorem}
    
    \begin{proof}
       First we consider the family $\mathcal{F}_2$ of rank-2 vector bundles given by a non trivial extension of type
    \begin{equation}
        \label{fam_estables2}
        0\rightarrow \mathcal{O}_X\rightarrow E_2\rightarrow I_Z(f)\rightarrow0,
    \end{equation}
    where $Z$ is a generic 0-dimensional subscheme of length $|Z|=c_2-2-e\gg0$.
Let us consider $D=C_0-f\in \Pic(X)$.

First of all, notice that  $(-1)^\alpha D$,  $K_X+f+D$, $K_X+f-D$ are non effective divisors for $\alpha\in\{0,1\}$ and $|Z|>2-e\geq \h^0\mathcal{O}_X(D+f)$. Therefore, by Lemma \ref{lemma_E2},
  \begin{equation}
  \label{condiciones2}
\h^0E_2(D)=0; \h^0E_2(K_X+f+D)=0; \h^2E_2(K_X+D)=0; \h^2E_2(-f+D)=0.      
  \end{equation}

Let us now consider the family $\mathcal{F}_3$ of rank-3 vector bundles on $X$ given by a non trivial extension of type
$$\mathcal{F}_3: 0\rightarrow E_2\rightarrow E_3\rightarrow \mathcal{O}_X(D)\rightarrow0,$$

\noindent where $E_2\in\mathcal{F}_2$. Using the exact sequence (\ref{fam_estables2}), we see that $\h^1E_2(-D)=|Z|>0$.
Finally consider the family $\mathcal{F}_4$ of rank-4 vector bundles on $X$ given by  a non trivial extension of type
\begin{equation}
\label{Frecursiva4}
    0\rightarrow E_3\rightarrow E_4\rightarrow \mathcal{O}_X(-D)\rightarrow0,
\end{equation}

\noindent where $E_3\in\mathcal{F}_3$.

By construction, $E_4$ is a rank-4 vector bundle with $c_1(E_4)=f$ and $c_2(E_4)=c_2$. Thus, Riemann-Roch Theorem allows us to see that $\ext^1(\mathcal{O}_X(-D),E_3)>0$.
Hence $\mathcal{F}_4\neq\emptyset$. 

Finally, let us see that $E_4$ is  simple and prioritary.
Since $(-1)^\alpha D$ and $K_X+f+2(-1)^{\alpha} D$ are non effective, by Lemma \ref{lema_recursivo}, conditions (\ref{condiciones2}) imply that 
$$\h^0E_3(D) =0;  \h^0E_3(K_X+f+D)=0; \h^2E_3(K_X+D)=0;\h^2E_3(-f+D) =0.$$

 Moreover, since $D$ is under assumptions of Lemma \ref{prop_ej1}, $E_3$ is a simple prioritary rank-3 vector bundle.
Putting altogether, $E_3$ is under conditions of Theorem \ref{thprior1} and hence  $E_4$ is a simple prioritary vector bundle.
Finally, by construction, $\h^0E_4\geq \h^0E_3\geq \h^0E_2\geq1$.

    \end{proof}
 \begin{theorem}
 \label{rango4}
    Let $X$ be a ruled surface over a non singular curve $C$ of genus $g\geq0$ such that $\overline{e}:=e-2(g-1)<0$ and $c_2\gg0$ an integer. Then there exists a rank-4 simple prioritary  vector bundle $E_4$ with $c_1(E_4)=f$ and $c_2(E_4)=c_2$ such that $\h^0E_4\geq1$.
\end{theorem}

      \begin{proof}
       The proof follows step by step as in the proof of Theorem \ref{thrk4''} but starting with the family $\mathcal{F}_2$ of rank-2 vector bundles $E_2$ given by a non trivial extension of type
    \begin{equation}
        \label{fam_estables'}
        0\rightarrow \mathcal{O}_X\rightarrow E_2\rightarrow I_Z(f)\rightarrow0,
    \end{equation}
    where $Z$ is a generic 0-dimensional subscheme of length $|Z|=c_2+e-4g+2\gg0$ and taking $D=C_0+(\overline{e}-1)f\in \Pic(X)$.
 \end{proof}

We end the section dealing with the case $c_1=C_1+\fm f$ which will arise as a consequence of the results of rank-3 simple prioritary vector bundles.

\begin{theorem}
   \label{th_rk4}
    Let $X$ be a ruled surface over a non singular curve $C$ of genus $g\geq0$ and   $c_1=C_0+\fm f\in \Pic(X)$ with $m=\deg(\fm)\in\{0,1\}$. Let  $c_2\gg0$ be an integer  and $3\leq l\leq5$ be the numerical class of $c_2-2-e$ module 3. Then, there exists a rank-4 simple prioritary vector bundle $E$ with $c_1(E)=c_1$ and $c_2(E)=c_2$ such that $\h^0E\geq k:=\frac{1}{3}(c_2-2-e-l)+1-g$.  
    
\end{theorem}

\begin{proof}
    Let us consider the family $\mathcal{G}_2$ of rank-2 vector bundles on $X$ given by a non-trivial extension of type $$0\rightarrow\mathcal{O}_X(\fb f)\rightarrow E_2\rightarrow I_{Z}(C_0+(\fm-\fb)f)\rightarrow0,$$
    where  $\fb \in \Pic(X)$ is such that $b=\deg (\fb)=\frac{1}{3}(c_2-2-e-l)$ and $Z$ is a generic 0-dimensional subscheme of length $|Z|=l$. Notice that, since $c_2\gg0$, we can assume $b>\max\{0,2(g-1)\}.$
    Let us fix $D= C_0-(\fb+1)f$.

      First of all, notice that $(-1)^\alpha [D+\fb f]$, $(-1)^\alpha[D+C_0+(\fm-\fb)f]$, $K_X+f+D+(-1)^\alpha[C_0+(\fm-\fb)f)]$ and $K_X+f+ D+(-1)^{\alpha}\alpha(\fb f)$ are non effective divisors for $\alpha\in\{0,1\}$. 
      Therefore,   by Lemma \ref{lema_E2'}, 
\begin{equation}
    \label{condiciones}
    \h^0E_2(D)=0; \h^0E_2(K_X+f+D)=0; \h^2E_2(K_X+D)=0; \h^2E_2(-f-D)=0. 
\end{equation}

Let us now consider the family $\mathcal{G}_3$ of rank-3 vector bundles on $X$ given by a non trivial extension of type
$$ 0\rightarrow E_2\rightarrow E_3\rightarrow \mathcal{O}_X(D)\rightarrow0,$$
where $E_2\in\mathcal{G}_2$.
Finally, consider the family $\mathcal{G}_4$ of rank-4 vector bundles on $X$ given by a non trivial extension of type
$$0\rightarrow E_3\rightarrow E_4\rightarrow \mathcal{O}_X(-D)\rightarrow0$$
where $E_3\in\mathcal{G}_3$.
Notice that by Riemann-Roch, $\ext^1(\mathcal{O}_X(-D),E_3)>0$ and hence $\mathcal{G}_4\neq\emptyset$. 
Moreover, by construction
$E_4$ is a rank-4 vector bundle with 
$    c_1(E_4)=C_0+\fm f$ and $c_2(E_4)= c_2.$

\noindent\textbf{Claim : $E_4$ is  simple and prioritary.}

\noindent\textbf{Proof of the Claim:}
Since $(-1)^\alpha D$ and $K_X+f+2(-1)^\alpha D$ are non effective divisors for $\alpha\in\{0,1\}$, conditions (\ref{condiciones}) imply by Lemma \ref{lema_recursivo} that 
$$ \h^0E_3(D) =0; \h^0E_3(K_X+f+D)=0; \h^2E_3(K_X+D)=0;\h^2E_3(-f-D) =0. $$

 Moreover, arguing as in the proof of Theorem \ref{prop_ej3}, we see that $E_3$ is a simple prioritary rank-3 vector bundle.

Therefore, $E_3$ is under conditions of Theorem \ref{thprior1} and hence,  $E_4$ is a simple prioritary vector bundle. 

Finally, by construction and the fact that $b>\max\{0,2(g-1)\}$, $$\h^0E_4\geq \h^0E_3\geq \h^0E_2\geq\h^0\mathcal{O}_X(\fb f)=\h^0\mathcal{O}_C(\fb)=b+1-g.$$

\end{proof}

\section{Higher rank prioritary bundles}

The goal of this section is to prove the existence of simple prioritary vector bundles of arbitrary high rank with at least certain number of sections. In general it is a difficult problem to construct indecomposable rank $r$ vector bundles on a smooth projective variety $X$ when $r$ is big compared with $\dim X$. We will achive our goal using a generalization of Cayley-Bacharach property and in addition we will see that the bundles that we construct are prioritary with a certain number of sections. To this end, we start with a technical Lemma that we will use later on.

\begin{lemma}
\label{lema_previo3}
    Let $X$ be a ruled surface over a non singular curve $C$ of genus $g\geq0.$
    Let $E$ be a rank $r$ vector bundle on $X$ given by a nontrivial extension of type
    \begin{equation}
    \label{fam_r}
    0\rightarrow \mathcal{O}_X(\fb f)\rightarrow E\rightarrow \bigoplus_{i=1}^{r-1}I_{Z_i}(D_i)\rightarrow0
\end{equation}
where for every $i$, 
$Z_i$ is a 0-dimensional subscheme.
Let us assume that:

i) for every $i\neq j$, $D_i-D_j$ is non effective,

ii) for every i, $\fb f-D_i$ is non effective, 

iii) for every $i\neq j$, $D_i-D_j+K_X+f$ is non effective,

iv) for every $i$, $\fb f-D_i+K_X+f$ is non effective, 

v) for every $i$, $\h^0I_{Z_i}(D_i-\fb f)=0$, 

vi) for every $i$, $\h^0I_{Z_i}(K_X+f+D_i-\fb f)=0$.

\noindent Then, $E$ is a simple  prioritary vector bundle.

\end{lemma}

\begin{proof}
First of all we will see that $E$ is simple.
Applying the functor $\Hom(-,E)$ to the exact sequence (\ref{fam_r}), we get
$$0\rightarrow \Hom(\bigoplus_{i=1}^{r-1}I_{Z_i}(D_i),E)\rightarrow \Hom(E,E)\rightarrow \Hom(\mathcal{O}_X(\fb f),E)\rightarrow\cdots$$

\noindent\textbf{Claim 1: $\Hom(\bigoplus_{i=1}^{r-1}I_{Z_i}(D_i),E)=0$}

\noindent\textbf{Proof of Claim 1:}
Since $\Hom(\bigoplus_{i=1}^{r-1}I_{Z_i}(D_i),E)\cong \bigoplus_{i=1}^{r-1}\Hom(I_{Z_i}(D_i),E)$, it is enough to prove that, for any $i$, $\Hom(I_{Z_{i}}(D_{i}),E)=0$.

Let us fix $i_0\in I:=\{1,\dots,r-1\}$.  Applying  the functor $\Hom(I_{Z_{i_0}}(D_{i_0}),-)$ to (\ref{fam_r}) we get the long exact sequence
\begin{equation}
\label{simple}
   0\rightarrow \Hom(I_{Z_{i_0}}(D_{i_0}),\mathcal{O}_X(\fb f))\rightarrow \Hom(I_{Z_{i_0}}(D_{i_0}),E)
   \rightarrow \Hom(I_{Z_{i_0}}(D_{i_0}),\bigoplus_{i=1}^{r-1}I_{Z_i}(D_i))
\end{equation}
\[\rightarrow
   \Ext^1(I_{Z_{i_0}}(D_{i_0}),\mathcal{O}_X(\fb f))\rightarrow\cdots\]

On one hand, by ii),

$\begin{array}{ll}
 \Hom(I_{Z_{i_0}}(D_{i_0}),\mathcal{O}_X(\fb f))    & \cong\Ext^2(\mathcal{O}_X(\fb f),I_{Z_{i_0}}(D_{i_0}+K_X)) \\
    & \cong \Ho^2 \mathcal{O}_X(D_{i_0}-\fb f+K_X)\\
    & \cong\Ho^0 \mathcal{O}_X(\fb f-D_{i_0})=0.
\end{array}$

On the other hand, 
$$\Hom(I_{Z_{i_0}}(D_{i_0}),\bigoplus_{i=1}^{r-1}I_{Z_i}(D_i))\cong \bigoplus_{i=1}^{r-1}\Hom(I_{Z_{i_0}}(D_{i_0}),I_{Z_i}(D_i)).$$

If $i=i_0$, by Lemma \ref{lema_IZ}, we get  $$\Hom(I_{Z_{i_0}}(D_{i_0}),I_{Z_i}(D_i))\cong\Hom(I_{Z_{i_0}},I_{Z_{i_0}})\cong K.$$

Assume $i\neq i_0$. Applying the functor $\Hom(I_{Z_{i_0}},-)$ to the exact sequence $$0\rightarrow I_{Z_{i}}\rightarrow \mathcal{O}_X\rightarrow \mathcal{O}_{Z_{i}}\rightarrow0$$ twisted by $\mathcal{O}_X(D_i-D_{i_0})$, we get

$$\dim \Hom(I_{Z_{i_0}},I_{Z_i}(D_i-D_{i_0}))\leq \dim \Hom(I_{Z_{i_0}}, \mathcal{O}_X(D_i-D_{i_0})).$$

On the other hand, 

$\begin{array}{ll}
   \hom(I_{Z_{i_0}}, \mathcal{O}_X(D_i-D_{i_0}))  &= \ext^2(\mathcal{O}_X(D_i-D_{i_0}), I_{Z_{i_0}}(K_X) )$$
$$ \\
   & =\h^0\mathcal{O}_X(D_i-D_{i_0})  =0,
\end{array}$

where the last equality follows from  assumption i). 
Putting altogether,  
$$\bigoplus_{i=1}^{r-1}\Hom(I_{Z_{i_0}}(D_{i_0}),I_{Z_i}(D_i))\cong K.$$

In addition,
     $$\hom(I_{Z_{i_0}}(D_{i_0}),\mathcal{O}_X(\fb f))= \h^2(I_{Z_{i_0}}(D_{i_0}+K_X-\fb f)) 
     =\h^0\mathcal{O}_X(\fb f-D_{i_0}) =0,$$

where the last inequality follows from ii).
Hence, the exact sequence (\ref{simple}) turns to be  $$0\rightarrow \Hom(I_{Z_{i_0}},E)\rightarrow K\overset{\gamma_{i_0}}{\rightarrow} \Ext^1(I_{Z_{i_0}}(D_{i_0}),\mathcal{O}_X(\fb f))\rightarrow \cdots,$$
which means that $$\Hom(I_{Z_{i_0}},E)\cong\ker(\gamma_{i_0}).$$
By construction, the extension $e=(e_1,\dots,e_{r-1})\in\Ext^1(\bigoplus_{i=1}^{r-1}I_{Z_i}(D_i),\mathcal{O}_X(\fb f))$ is defined by nontrivial extensions $e_i\in\Ext^1(I_{Z_i}(D_i),\mathcal{O}_X(\fb f))$ and $\gamma_{i_0}$ is the map that sends $1\in K$ to the nontrivial extension $e_{i_0}\in \Ext^1(I_{Z_{i_0}}, \mathcal{O}_X(\fb f))$. Thus, $\gamma_{i_0}$ is injective and hence $$\Hom(I_{Z_{i_0}},E)\cong\ker(\gamma_{i_0})=0.$$

Consider the exact sequence (\ref{fam_r}) twisted by $\mathcal{O}_X(-\fb f)$ and  taking cohomology we get
$$0\rightarrow \Ho^0\mathcal{O}_X\rightarrow \Ho^0E(-\fb f)\rightarrow \bigoplus_{i=1}^{r-1}\Ho^0I_{Z_i}(D_i-\fb f)\rightarrow\cdots$$

\noindent Thus, by iv),  $$\Hom(\mathcal{O}_X(\fb f), E)\cong \Ho^0E(-\fb f)\cong \Ho^0\mathcal{O}_X\cong K.$$
Therefore, $1\leq\dim\Hom(E,E)\leq1$ and hence $E$ is simple.

\vspace{0.2cm}

Let us now see that $E$ is prioritary.
Applying the functor $\Hom(E,-)$ to the exact sequence (\ref{fam_r}) twisted by $\mathcal{O}_X(-f)$, we obtain
\begin{equation}
    \label{prioritary}
    \cdots\rightarrow \Ext^2(E,\mathcal{O}_X((\fb-1) f))\rightarrow \Ext^2(E,E(-f))\rightarrow \Ext^2(E,\bigoplus_{i=1}^{r-1}I_{Z_i}(D_i-f))\rightarrow0.
\end{equation}

\noindent\textbf{Claim 2: $\Ext^2(E,\mathcal{O}_X((\fb-1) f))=0$. }

\noindent\textbf{Proof of Claim 2:}
By duality, $\ext^2(E,\mathcal{O}_X((\fb-1) f))= \h^0E(K_X-(\fb-1)f).$
 Twisting (\ref{fam_r}) by $\mathcal{O}_X(K_X-(\fb-1) f)$ and taking cohomology we get 
$$0\rightarrow \Ho^0\mathcal{O}_X(K_X+f)\rightarrow \Ho^0E(K_X-(\fb-1) f)\rightarrow \bigoplus_{i=1}^{r-1}\Ho^0I_{Z_i}(D_i+K_X-(\fb-1) f)\rightarrow\cdots$$

\noindent By v), $\Ho^0\mathcal{O}_X(K_X+f)=\Ho^0I_{Z_i}(D_i+K_X-(\fb-1) f)=0,$ which implies  $$ \Ho^0E(K_X-(\fb-1) f)=0.$$

\noindent\textbf{Claim 3: $\Ext^2(E,\bigoplus_{i=1}^{r-1}I_{Z_i}(D_i-f))=0.$ }

\noindent\textbf{Proof of Claim 3:}
Notice that $$\Ext^2(E,\bigoplus_{i=1}^{r-1}I_{Z_i}(D_i-f))\cong \bigoplus_{i=1}^{r-1} \Ext^2(E,I_{Z_i}(D_i-f)).$$

Then, it is sufficient to prove that, for a fixed $i_0\in I:=\{1,\dots,r-1\}$, we have that $\Ext^2(E,I_{Z_{i_0}}(D_{i_0}-f))=0$. Applying the functor $\Hom(-,I_{Z_{i_0}}(D_{i_0}-f) )$ to (\ref{fam_r}), we get
\begin{equation}
    \label{prioritary2}
    \cdots \Ext^2(\bigoplus_{i=1}^{r-1}I_{Z_i}(D_i),I_{Z_{i_0}}(D_{i_0}-f) )\rightarrow \Ext^2(E,I_{Z_{i_0}}(D_{i_0}-f) )\rightarrow \Ext^2(\mathcal{O}_X(\fb f),I_{Z_{i_0}}(D_{i_0}-f) )\rightarrow0
\end{equation}

Let us first check that $$\Ext^2(\bigoplus_{i=1}^{r-1}I_{Z_i}(D_i),I_{Z_{i_0}}(D_{i_0}-f) )=0.$$

Since for any $i_0$, we have the exact sequence

$$\cdots\rightarrow \Ext^2(\mathcal{O}_X(D_i),I_{Z_{i_0}}(D_{i_0}-f)) \rightarrow \Ext^2(I_{Z_i}(D_i),I_{Z_{i_0}}(D_{i_0}-f))\rightarrow0$$
and by iii) $$\ext^2(\mathcal{O}_X(D_i),I_{Z_{i_0}}(D_{i_0}-f))=\h^0\mathcal{O}_X(D_i-D_{i_0}+K_X+f)=0, $$ we get $$\Ext^2(I_{Z_i}(D_i),I_{Z_{i_0}}(D_{i_0}-f))=0,$$ which implies that $$\Ext^2(\bigoplus_{i=1}^{r-1}I_{Z_i}(D_i),I_{Z_{i_0}}(D_{i_0}-f) )=0.$$

Finally, by iv),
$$\ext^2(\mathcal{O}_X(\fb f),I_{Z_{i_0}}(D_{i_0}-f) )=\h^0\mathcal{O}_X(\fb f+f-D_{i_0}+K_X)=0.$$ 

Putting altogether, by (\ref{prioritary}) and (\ref{prioritary2}), we see that $E$ is prioritary.  
\end{proof}

\begin{theorem}
\label{th_rgr1}
     Let $X$ be a ruled surface over a non singular curve $C$ of genus $g\geq0$ such that $\overline{e}:=e- 2(g-1)>0$. Fix $c_2\gg0$ and $r\geq4$ integers, and  $c_1= sC_0+\mathfrak{t}f\in \Pic(X)$ with $0\leq s\leq r-1$ and $t:=\deg(\mathfrak{t})\geq\max\{-r+2,2g-r\}$.
    Then, there exists a rank $r$ simple prioritary  vector bundle with $c_1(E)=c_1$ and $c_2(E)=c_2$ such that $\h^0E\geq k:=r-1+t-g$.
\end{theorem}

\begin{proof}
First of all, let $\mathfrak{b}\in\Pic(X)$ be a divisor of degree $b=r-2+t$ and consider  $r-1$ divisors $D_i\in\Pic(X)$ such that: 

$\begin{array}{ll}
D_1& \equiv  -(\sum_{j=2}^{r-1}j)C_0+(r-1+t)f,\\
    D_i & \equiv iC_0+(r-i)f, \thinspace \thinspace \text{for} \thinspace 2\leq i\leq r-2 \thinspace\thinspace \text{and} \\
   D_{r-1} & \equiv  (r-1+s)C_0-(b+\sum_{j=2}^{r-1}j)f.  
\end{array}$

Let us define 

$\begin{array}{ll}
   f(b):=  & r-2+\sum_{i=2}^{r-2}\max\{0,\chi\mathcal{O}_X(\fb f-D_i), (i+1)(r-i-b+1)\}+\sum_{i<j}D_i\cdot D_j \\
  & +(\fb f)(\sum_{i=1}^{r-1}D_i) +\max\{0,\chi\mathcal{O}_X(\fb f- D_{r-1})\}
\end{array}$

\noindent and notice that, since $c_2\gg0$ and $f(b)$ is independent of $c_2$, we can assume   $c_2>f(b)$. Let $Z_i$ be a generic 0-dimensional subschemes of length 

$\begin{array}{ll}
    |Z_i|&=\max\{0,\chi\mathcal{O}_X(\fb f-D_i), (i+1)(r-i-b+1)\}+1, \thinspace \thinspace\text{for} \thinspace 2\leq i\leq r-2, \\
   |Z_{r-1}| & = c_2-1-\sum_{i=2}^{r-2}|Z_i|-\sum_{i<j}D_i\cdot D_j -(\fb f)(\sum_{i=1}^{r-1}D_i),  
\end{array}$

\noindent and satisfying  $Z_i\cap Z_j=\emptyset$ for $i\neq j$.

\vspace{0.2cm}

Notice that, since $c_2>f(b)$, we get by definition that 
\begin{equation}
\label{long_r-12}
    |Z_{r-1}|=  c_2-1-\sum_{i=2}^{r-2}|Z_i|-\sum_{i<j}D_i\cdot D_j -(\fb f)(\sum_{i=1}^{r-1}D_i)>
\end{equation}
$$f(b)-1-\sum_{i=2}^{r-2}|Z_i|-\sum_{i<j}D_i\cdot D_j -(\fb f)(\sum_{i=1}^{r-1}D_i)=\max\{0,\chi\mathcal{O}_X(\fb f- D_{r-1})\}.$$

Keeping the above notation, let us consider   the family $\mathcal{F}$ of rank $r$ vector bundles $E$ on $X$ with $c_1(E)=c_1$ and $c_2(E)=c_2$ given by a nontrivial extension of type
\begin{equation}
    \label{fam_r2}
    0\rightarrow \mathcal{O}_X(\fb f)\rightarrow E\rightarrow \bigoplus_{i=1}^{r-1}I_{Z_i}(D_i)\rightarrow0
\end{equation}
where  $Z_1\in X$ is a point such that $Z_1\not\in Z_i$ for any $2\leq i\leq r-1$.

\noindent\textbf{Claim : $\mathcal{F}\neq\emptyset$.}

\noindent\textbf{Proof of the Claim :}
First of all let us see that the dimension of the space of extensions of type (\ref{fam_r2}) is positive. Notice that, for any $i$, we have that
$$\ext^1(I_{Z_i}(D_i),\mathcal{O}_X(\fb f))\geq-\chi I_{Z_i}(D_i+K_X-\fb f)=-\chi I_{Z_i}(\fb f-D_i).$$

Notice that, since $b<r-1+t$,  $\fb f-D_1$ and $D_1-\fb f+K_X$ are non effective divisors. Therefore, 
$-\chi I_{Z_1}(\fb f-D_1)=1-\chi\mathcal{O}_X(\fb f-D_1)>0.$

For $2\leq i\leq r-2,$ by construction $$-\chi I_{Z_i}(\fb f-D_i)=|Z_{i}|-\chi\mathcal{O}_X(\fb f-D_i)>0.$$

Finally, for $i=r-1$, $$-\chi I_{Z_{r-1}}(\fb f-D_i)=|Z_{r-1}|-\chi\mathcal{O}_X(\fb f-D_i)>0,$$
where the last inequality follows from (\ref{long_r-12}).

Therefore $\ext^1(I_{Z_i}(D_i),\mathcal{O}_X(\fb f))>0$ for every $i$ and hence the dimension of the space of extensions of type (\ref{fam_r2}) is positive.

Let us now see that the couple  $(\mathcal{O}_X(D_i-\fb f+K_X),Z_i)$ satisfies the Cayley-Bacharach property for every $i$.
 First of all, since the divisor $D_1-\fb f+K_X$ is non effective,  $$\h^0I_{Z_1}(D_1-\fb f+K_X)=0.$$ For $2\leq i\leq r-2$, since   $Z_i$ is generic and $$\h^0\mathcal{O}_X(D_i-\fb f+K_X)< i(r-i-b+1)<|Z_i|,$$ we have  $\h^0I_{Z_i}(D_i-\fb f+K_X)=0$. Finally, since the divisor $D_{r-1}-\fb f+K_X$ is non effective, we also have $\h^0I_{Z_{r-1}}(D_{r-1}-\fb f+K_X)=0$. Therefore, $(\mathcal{O}_X(D_i-\fb f+K_X),Z_i)$ satisfies the Cayley-Bacharach property for every $i$ and,  by \cite[Proposition 2.2]{Qin}, this guaranties that $E$ is a rank $r$ vector bundle with $c_1(E)=sC_0+\mathfrak{t}f$ and $$c_2(E)=1+\sum_{i=2}^{r-1}|Z_i|+\sum_{i<j}D_i\cdot D_j+(\fb f)(\sum_{i=1}^{r-1}D_i)=c_2.$$

Hence $\mathcal{F}$ is non empty.

Notice that the divisors $D_i$ verify all the assumptions of Lemma \ref{lema_previo3}. In fact, i)-iv) are straigthfoward and also v) and vi) for the cases $i=1$ and $i=r-1$. For $2\leq i\leq r-2$, to see v) and vi) we have two possibilities. If $D_i-\fb f$ and $K_X+f+D_i-\fb f$ are non effective divisors, we directly get $\h^0I_Z(D_i-\fb f)=0$ and  $\h^0I_Z(K_X+f+D_i-\fb f)=0$. Otherwise we use the fact that $Z_i$ is generic of length $|Z_i|>\h^0\mathcal{O}_X(D_i-\fb f+K_X+f)$ and $|Z_i|>\h^0\mathcal{O}_X(D_i-\fb f)$  since $$\h^0\mathcal{O}_X(D_i-\fb f)\leq (i+1)(r-i-b+1)$$ and $$\h^0\mathcal{O}_X(D_i-\fb f+K_X+f)\leq (i-1)(r-i-b-\overline{e}+2).$$
Hence, $E$ is a simple prioritary vector bundle.

Moreover, since $b=r-2+t\geq \max\{2(g-1),0\}$, by (\ref{fam_r2}) we get $$\h^0E\geq \h^0\mathcal{O}_X(\fb f)=b+1-g=r-1+t-g.$$

\end{proof}

\begin{remark}
    The case $\overline{e}=0$ can be proved analogously for $b=r-3+t$ We consider the same divisors $D_i$ as in Proof of Theorem \ref{th_rgr1},  we define the function  
    
   $ \begin{array}{ll}
        f(b)= &r-2+\sum_{i=2}^{r-2}\max\{\chi\mathcal{O}_X(\fb f-D_i), (i+1)(r-i-b+1), (i-1)(r-i-b+2)\}  \\
        & +\sum_{i<j}D_i\cdot D_j+(\fb f)(\sum_{i=1}^{r-1}D_i)+\max\{0,\chi\mathcal{O}_X(\fb f-D_{r-1})\}
    \end{array}$
and we can assume that $c_2>f(b)$. For $1\leq i\leq r-2$, we consider  generic 0-dimensional subschemes $Z_i$ of length $$|Z_i|=\max\{0,(i+1)(r-i-b+1), (i-1)(r-i-b+2)\}+1.$$ Finally, we take $Z_1$ and $Z_{r-1}$ as in Proof of Theorem \ref{th_rgr1}. 
\end{remark}

\begin{remark}
   The case $\overline{e}=-1$ can be proved analogously to Theorem \ref{th_rgr1} for $b=r-6+t$. We define the divisors $D_i$  as
    
   $\begin{array}{ll}
D_1& \equiv  -(\sum_{j=2}^{r-1}j)C_0+(r-3+t)f\\
    D_i & \equiv iC_0+(r-3i)f, \thinspace \text{for} \thinspace 2\leq i\leq r-2 \\
   D_{r-1} & \equiv  (r-1+s)C_0-(b+\sum_{j=2}^{r-1}(r-3j))f.  
\end{array}$

\noindent and the function 

$\begin{array}{ll}
   f(b):=  & r-2+\sum_{i=2}^{r-2}\max\{0,\chi\mathcal{O}_X(\fb f-D_i), (i+1)(r-3i-b+1)\}+\sum_{i<j}D_i\cdot D_j \\
  & +(\fb f)(\sum_{i=1}^{r-1}D_i) +\max\{0,\chi\mathcal{O}_X(\fb f- D_{r-1})\}
\end{array}$
Finally, we consider $Z_i$  generic 0-dimensional subschemes of length 

$\begin{array}{ll}
    |Z_i|&=\max\{0,\chi\mathcal{O}_X(\fb f-D_i), (i+1)(r-3i-b+1),(i-1)(r-3i-b+3)\}+1, \thinspace \thinspace\text{for} \thinspace 2\leq i\leq r-2,
\end{array}$
and $Z_1$ and $Z_{r-1}$ as in Theorem \ref{th_rgr1}.
\end{remark}

Let us now consider the case $\overline{e}<-1$.

\begin{theorem}
\label{th_rg_r2}
    Let $X$ be a ruled surface over a non singular curve $C$ of genus $g\geq0$ such that $\overline{e}:=e- 2(g-1)<-1$. Fix $r\geq4$ and $c_2\gg0$  integers and  $c_1= sC_0+\mathfrak{t}f\in \Pic(X)$ with $0\leq s \leq r-1$ and $$t:=\deg(\mathfrak{t})>\max\{0,2g-2\}+\frac{[(r-1)(r-2)+2]\overline{e}}{2}.$$
    Then there exists a rank $r$ simple prioritary vector bundle $E$ with $c_1(E)=c_1$ and $c_2(E)=c_2$, such that $\h^0E\geq t-g-[\frac{(r-1)(r-2)+2}{2}]\overline{e}$.
\end{theorem}

\begin{proof}
Since the proof is similar to the one of Theorem \ref{th_rgr1}, we highlight the main differences and left the details to the reader.
First of all let us consider $\fb\in\Pic(X)$ of degree $$b=t-1-[\frac{(r-1)(r-2)+2}{2}]\overline{e}$$ 
 and  define the following $r-1$ divisors $D_i\in\Pic(X)$ 

$\begin{array}{ll}
    D_i &\equiv i(C_0+\overline{e})f, \thinspace \text{for} \thinspace\thinspace 1\leq i\leq r-2 \\
   D_{r-1} &= c_1-\fb f-\sum_{i=1}^{r-2}D_i  
\end{array}$

and  $$f(b):=r-2+\sum_{i<j}D_i\cdot D_j +(\fb f)(\sum_{i=1}^{r-1}D_i)+\max\{0,\chi\mathcal{O}_X(\fb f- D_{r-1}),0\}.$$ Notice that, since $c_2\gg0$,
 $c_2>f(b)$. Let us consider $Z_{r-1}$  a 0-dimensional subscheme of length 

$\begin{array}{ll}
   |Z_{r-1}| &= c_2-(r-2)-\sum_{i<j}D_i\cdot D_j -(\fb f)(\sum_{i=1}^{r-1}D_i).  
\end{array}$

\vspace{0.2cm}

Notice that, since $c_2>f(b)$, 
$ |Z_{r-1}|>\max\{0,\chi\mathcal{O}_X(\fb f- D_{r-1})\}.$

Keeping the above notation, consider  the family $\mathcal{F}$  of rank $r$ vector bundles $E$ on $X$ given by a nontrivial extension of type
\begin{equation}
    \label{fam_r3}
    0\rightarrow \mathcal{O}_X(\fb f)\rightarrow E\rightarrow [\bigoplus_{i=1}^{r-2}I_{P_i}(D_i)]\oplus I_{Z_{r-1}}(D_{r-1})\rightarrow0,
\end{equation}
where $\fb\in \Pic(C)$ is a divisor of degree $b$ and $P_1,\dots,P_{r-2}\in X$ are different points such that $$Z_{r-1}\cap \{P_1,\dots,P_{r-2}\}=\emptyset.$$

We check that $\mathcal{F}$ is non-empty and $E$ is a rank r vector bundle on $X$ with $c_1(E)=sC_0+\mathfrak{t}f$ and $c_2(E)=c_2>f(b)$
and by Lemma \ref{lema_previo3} it is simple and prioritary. Finally, since $b\geq \max\{0,2g-2\}$, by (\ref{fam_r3}) we get $$\h^0E\geq \h^0\mathcal{O}_X(\fb f)=b+1-g=t-g-[\frac{(r-1)(r-2)+2}{2}].$$ 
\end{proof}

\begin{remark}




Notice that for the case $c_1=C_0+\fm f$ and $r=4$, Theorem \ref{th_rk4} gives us a stronger result than the Theorems of this section.
       
\end{remark}


\begin{thebibliography}{99}



\bibitem{theta}
M. Aprodu and G. Casnati, L. Costa, R. M. Miró-Roig, M. Teixidor i Bigas: \textit{Thetha divisors and Ulrich bundles on geometrically ruled surfaces}; Ann. Mat. Pura Appl. (4) 199, No. 1, 199--216 (2020).



\bibitem{Arbarello}E. Arbarello, M. Cornalba, P.A. Griffiths, J. Harris, Geometry of Algebraic Curves. Vol. I., in: Grundlehren der Mathematischen Wissenschaften
(Fundamental Principles of Mathematical Sciences), vol. 267, Springer-Verlag, New York, 1985.

\bibitem{coskun}I. Coskun, J. Huizenga and H. Nuer, \textit{Brill-Noether Theory of moduli spaces of sheaves on surfaces}, Preprint https://arxiv.org/pdf/2306.11033.pdf, June 2023.


\bibitem{nuestro} L.Costa, I. Macías,  \textit{Brill-Noether Theory of vector bundles on ruled surfaces}, Mediterr. J. Math. 21, No. 3, Paper No. 118, 22 p. (2024).

\bibitem{CM2} L.Costa, I. Macías,  \textit{Brill-Noether Theory of prioritary bundles}, work in progress. 


\bibitem{friedman}
 R. Friedman, \textit{Algebraic surfaces and holomorphic vector bundles}. New York, NY: Springer (1998)


\bibitem{hartshorne}
 R. Hartshorne: \textit{Algebraic geometry}. G.T.M. 52, Springer (1977).


\bibitem{Hirschowitz}A. Hirschowitz, Y. Laszlo, \textit{Fibrés génériques sur le plan projectif}, Math Annalen 297:85-102, (1993).


\bibitem{Pedchenko}D. Pedchenko \textit{The Picard group of the moduli space of sheaves on a quadric surface}, Int. Math. Res. Not. 2022, No. 23, 18676--18765 (2022). 




\bibitem{Qin}W.-P. Li and Z. Qin, \textit{Stable vector bundles on algebraic surfaces}, Trans. Am. Math. Soc. 345, No. 2, 833--852 (1994; Zbl 0823.14028)

 \bibitem{Walter}Charles H. Walter, \textit{Irreducibility of Moduli Spaces of Vector bundles on Birrationally Ruled Surfaces}.






\end{thebibliography}
    \end{document}